\documentclass[10pt]{amsart}
\usepackage[left=1in, right=1in]{geometry} 
\usepackage[utf8]{inputenc}
\usepackage{amssymb}
\usepackage{amscd,amsfonts,amsmath,amssymb, bbm,dsfont, comment}
\usepackage{amsmath}
\usepackage{url}
\usepackage{amsfonts}
\usepackage{amssymb}\usepackage{enumerate,epsf,fancyhdr,float,graphicx,tabularx}
\usepackage{latexsym,mathrsfs,multirow}
\usepackage{wasysym}
\usepackage{xypic}
\usepackage[all]{xy}\usepackage[OT2,T1]{fontenc}
\usepackage[modulo,mathlines,displaymath, running]{lineno}
\usepackage{amsmath,mathrsfs,enumerate,wasysym}
\usepackage{xypic,hhline}
\usepackage[all]{xy}
\usepackage[OT2,T1]{fontenc}
\usepackage[modulo,mathlines,displaymath]{lineno}
\usepackage{tikz,tikz-cd}
\usepackage{dashrule}
\usetikzlibrary{matrix}
\usepackage[modulo,mathlines,displaymath]{lineno}

\newtheorem{theorem}{Theorem}[section]

\DeclareSymbolFont{cyrletters}{OT2}{wncyr}{m}{n}\DeclareMathSymbol{\Sha}{\mathalpha}{cyrletters}{"58}
\DeclareMathSymbol{\FSha}{\mathalpha}{cyrletters}{"11}

\renewcommand{\phi}{{\varphi}}

\renewcommand{\geq}{\geqslant}
\renewcommand{\leq}{\leqslant}

\newcommand{\HIwcyc}{H^1_{\mathrm{Iw},\mathrm{cyc}}}
\newcommand{\Dcrisp}{\mathbb{D}_{\mathrm{cris},p}}
\newcommand{\Fil}{\mathrm{Fil}}
\newcommand{\cris}{\mathrm{cris}}

\newcommand{\links}{\left(\begin{array}{cc}}
\newcommand{\rechts}{\end{array}\right)}
\newcommand{\bai}{\left[\begin{array}{cc}}
\newcommand{\dai}{\end{array}\right]}
\newcommand{\hidari}{\left(\begin{array}{c}}
\newcommand{\migi}{\end{array}\right)}

\newcommand{\Q}{\mathbb{Q}}

\newcommand{\Z}{\mathbb{Z}}

\newcommand{\calH}{\mathcal{H}}
\newcommand{\calI}{\mathcal{I}}

\newcommand{\calL}{\mathcal{L}}
\newcommand{\calM}{\mathcal{M}}

\newcommand{\calX}{\mathcal{X}}

\newcommand{\Gal}{\operatorname{Gal}}

\newcommand{\Hom}{\operatorname{Hom}}

\newcommand{\rank}{\operatorname{rank}}
\newcommand{\coker}{\operatorname{coker}}

\newcommand{\GL}{\operatorname{GL}}
\newcommand{\ord}{\operatorname{ord}}
\newcommand{\col}{\operatorname{Col}}

\newcommand{\Sel}{\operatorname{Sel}}

\newcommand{\uI}{\underline{I}}
\newcommand{\uJ}{\underline{J}}

\renewcommand{\contentsname}{Contents\\{\footnotesize\normalfont(A table
of contents should normally not be included)}}

\newtheorem{auxiliary proposition}[theorem]{Auxiliary Proposition}

\newtheorem{definition}[theorem]{Definition}

\newtheorem{lemma}[theorem]{Lemma}
\newtheorem{main conjecture}[theorem]{Main Conjecture}
\newtheorem{main theorem}[theorem]{Main Theorem}
\newtheorem{modesty proposition}[theorem]{Modesty Proposition}

\newtheorem{open problem}[theorem]{Open Problem}

\newtheorem{proposition}[theorem]{Proposition}

\newtheorem{remark}[theorem]{Remark}

\newtheorem{convergence lemma}[theorem]{Convergence Lemma}
\newtheorem{corrected lemma}[theorem]{Corrected Lemma}
\newtheorem{growth lemma}[theorem]{Growth Lemma}
\newtheorem{coefficient lemma}[theorem]{Integrality Lemma}
\newtheorem{interpolation lemma}[theorem]{Interpolation Lemma}
\newtheorem{kernel lemma}[theorem]{Kernel Lemma}
\newtheorem{limit lemma}[theorem]{Limit Lemma}
\newtheorem{tandem lemma}[theorem]{Modesty Lemma}
\newtheorem{zero-finding lemma}[theorem]{Zero-Finding Lemma}

\title[~]{On signed Mordell-Weil groups for abelian varieties. }
\author{Jishnu Ray}

\address[Ray]{Harish Chandra Research Institute, A CI of Homi Bhabha National Institute,  Chhatnag Road, Jhunsi, Prayagraj (Allahabad) 211 019 India}
\email{jishnuray@hri.res.in; jishnuray1992@gmail.com}

\begin{document}

\maketitle
\begin{abstract}
 In this short note, we work in the general framework of supersingular abelian varieties defined over $\Q$.  Using Coleman maps constructed by Büyükboduk--Lei, we define some  objects called ``the multi-signed Mordell-Weil groups" for supersingular abelian varieties, make comments on the structure of the dual of these groups as an Iwasawa module  and show a (weak) control theorem. This recovers the case of elliptic curves over $\Q$ non-ordinary at the prime $p$ with $a_p=0$ studied by Antonio Lei. Using the multi-signed Mordell-Weil groups we define what we call ``the multi-signed Tate-Shafarevich groups" along the cyclotomic tower of $\Q$. Finally we pose some open questions related to our newly defined objects and make a remark on the asymptotic growth of these  multi-signed Tate-Shafarevich groups along the cyclotomic tower using an idea of Meng Fai Lim.
\end{abstract}
\section{Introduction}
Let $p$ be an odd prime, $E$ be an elliptic curve over $\Q$ with supersingular reduction at $p$. Let $\Q_\infty$ be the cyclotomic $\Z_p$-extension of $\Q$ with Galois group $\Gamma$. Write $\Q_{(n)}$ for the unique sub-extension of degree $p^n$ with Galois group $\Gal(\Q_{(n)}/\Q)=\Gamma_n$. Let $\Lambda=\Z_p[[\Gamma]]$ be the Iwasawa algebra of $\Gamma$.

 Let $A$ be a $g$-dimensional abelian variety over $\Q$ (its  $p$-adic Tate module is of $\Z_p$-rank $2g$)   with good supersingular reduction at $p$ and let $A^\vee$ be the dual abelian variety. 
For any subset $\uI \in \{1,...,2g\}$ of cardinality $g$, using the theory of Coleman maps, we define the $\uI$-Mordell-Weil group $\calM_{\uI}(A^\vee/\Q_{(n)})$ and the $\uI$-Tate-Shafarevich group $\FSha_{\uI}(A^\vee/\Q_{(n)})$ (see Section \ref{sec:signed_MW}) which fits into the following exact sequence 
$$0 \rightarrow \calM_{\uI}(A^\vee/\Q_{(n)}) \rightarrow \Sel_{\uI}(A^\vee/\Q_{(n)}) \rightarrow \FSha_{\uI}(A^\vee/\Q_{(n)}) \rightarrow 0,$$
where the middle term is the $\uI$-Selmer group defined as in \cite[Section 1.6]{ponsinet}. 

The definitions of the $\uI$-Mordell-Weil group $\calM_{\uI}(A^\vee/\Q_\infty)$ and  $\Sel_{\uI}(A^\vee/\Q_\infty)$
(and hence that of the $\uI$-Tate-Shafarevich group $\FSha_{\uI}(A^\vee/\Q_\infty)$) depend on a choice of a Hodge-compatible basis of the associated Dieudonné module (see Definition \ref{Hodge-compatible}). 
We show that if a certain matrix  is block anti-diagonal the groups $\calM_{\uI}(A^\vee/\Q_\infty)$ and $\Sel_{\uI}(A^\vee/\Q_\infty)$ are canonically defined for two indexing sets, namely $\uI=\{1,...,g\}$ and $\uI=\{g+1,...,2g\}$ (see Proposition \ref{prop:canonicaldefn}). Hence the groups 
$\calM_{\uI}(A^\vee/\Q_{(n)})$ and $\FSha_{\uI}(A^\vee/\Q_{(n)})$ are also canonically defined.
When the abelian variety is an elliptic curve (i.e. $g=1$), then these groups exactly coincide with their canonically defined plus and minus variants, namely the groups $\calM^\pm(E/\Q_{(n)})$  and $\FSha^\pm(E/\Q_{(n)})$ define by Lei in \cite[definition 5.2]{Lei2023plus_minus_MW}.
Our goal in this short note is to first define these objects, reviewing results from the existing literature, and to prove a (weak) control theorem for  $\calM_{\uI}(A^\vee/\Q_{(n)})$  as $n$ grows (Theorem \ref{thm:I_MW}). We also make some remarks on the algebraic structure of the Pontryagin dual of $\calM_{\uI}(A^\vee/\Q_{\infty})$ (see Proposition \ref{thm:calM_I}). These are generalizations of results in \cite{Lei2023plus_minus_MW} given for elliptic curves. Finally, using an idea by Meng Fai Lim, we comment on the asymptotic growth of $\FSha_{\uI}(A^\vee/\Q_{(n)})$ as $n$ varies and indicate few
 open questions.

\bigbreak
\section*{Acknowledgement}
The author thanks Antonio Lei, Meng Fai Lim and Christian Wuthrich for several helpful conversations. He also thanks the referee for his/her valuable comments. This project is supported by the Inspire research grant, DST, Govt. of India. 
\bigbreak
\section{Signed Selmer groups over the cyclotomic tower}

Let $A$ be a $g$-dimensional abelian variety defined over $\Q$ with good supersingular reduction at $p$. Throughout the article we will need the following notations.

\begin{itemize}
    \item For $n \geq 1$, let $A[p^n]$ be the $p^n$-torsion points of $A(\overline{\Q})$ and let $$A[p^\infty]=\cup_nA[p^n].$$ \smallbreak
    \item  Let $T$ be the $p$-adic Tate module associated with the abelian variety $A$ with a continuous  action of the absolute Galois group $G_{\Q}$. It is a free $\Z_p$-module of rank $2g$.\smallbreak
    \item Let $V=T \otimes_{\Z_p} \Q_p$, $V$ is a crystalline $G_{\Q_p}$-representation with Hodge-Tate weights $0$ and $1$, both with multiplicity $g$.\smallbreak
    \item For $i \geq 0$, the Iwasawa cohomology group $H^i_{\mathrm{Iw},\mathrm{cyc}}(\Q_p,T)$ is defined as the projective limit of $H^i(\Q(\mu_{p^n})_{v_n},T)$ relative to the corestriction maps. Here $v_n$ is the unique prime of $\Q(\mu_{p^n})$ above $p$.\smallbreak
    \item The Dieudonné module $\Dcrisp(T)$ associated to $T$ (see \cite{berger04})  is a free $\Z_p$-module of rank $2g$.\smallbreak
     \item As mentioned in the introduction, $A^\vee$ will be the dual abelian variety and for a $\Z_p$-module $M$,  $M^\land$ will be  the Pontryagin dual $\Hom(M,\Q_p/\Z_p)$.
\end{itemize}
 \smallbreak
 The module $\Dcrisp(T)$ is equipped with a Frobenius after tensoring with $\Q_p$ and a filtration of $\Z_p$-modules $(\Fil^i\Dcrisp(T))_{i \in \Z}$ satisfying
$$
\Fil^i \mathbb{D}_{\cris,p}(T) = \begin{cases} 0 & \text{if $i \geq 1$,} \\ \mathbb{D}_{\cris,p}(T) & \text{if $i \leq -1$.} \end{cases}
$$
\begin{definition}\label{Hodge-compatible}
A $\Z_p$-basis $\{u_1,\ldots,u_{2g} \}$ of $\mathbb{D}_{\cris,p}(T)$ is called Hodge-compatible if 
 $\{ u_1,\ldots, u_{g} \}$ is a $\Z_p$-basis of $\Fil^0\mathbb{D}_{\cris,p}(T)$. 
\end{definition}
We mention that such a choice of basis of $\mathbb{D}_{\cris,p}(T)$ can always be made and we choose one such basis once and for all.
The matrix of the Frobenius $\varphi$ with respect to this basis is of the form
$$
C_{\varphi,p} = C_p\left[
\begin{array}{c|c}
I_{g} & 0 \\
\hline
0 & \frac{1}{p} I_{g}
\end{array}
\right]
$$
for some $C_p \in \mathrm{GL}_{2g}(\Z_p)$ and  the identity $g \times g$ matrix $I_g$.
Following \cite[Definition 2.4]{BL17}, for $n \geq 1$, we can define
$$
C_{p,n} := \left[
\begin{array}{c|c}
I_{g} & 0 \\
\hline
0 & \Phi_{n}(1+X) I_{g}
\end{array}
\right]
C_p^{-1}
 \text{ and } M_{p,n}:= (C_{\varphi,p})^{n+1}C_{p,n} \cdots C_{p,1}, \hfill
$$
where $\Phi_{n}(1+X)$ is the $p^n-$th cyclotomic polynomial in $1+X.$
Recall that as $p$ is odd, $$\Gal(\Q(\mu_{p^\infty})/\Q) \cong \Delta \times \Gamma$$
where $\Delta$ is a finite group of order $p-1$.
Set $\calH=\Q_p[\Delta]\otimes_{\Q_p} \calH(\Gamma)$ where $\calH(\Gamma)$ is the set of elements $f(\gamma-1)$ with $\gamma \in \Gamma$ and $f(X)\in \Q_p[[X]]$ is convergent on the $p$-adic open unit disk. Let $$\calL_{T,p}: \HIwcyc(\Q_p,T) \rightarrow \calH \otimes \Dcrisp(T)$$
be the Perrin-Riou's big logarithm map (see \cite[Definition 3.4]{LLZ0.5}). It ``interpolates'' Kato's dual exponential maps \cite[II, \S 1.2]{Kato93_approach}
$$\mathrm{exp}^*_{p,n}: H^1(\Q_p(\mu_{p^n}),T) \rightarrow \Q_p(\mu_{p^n}) \otimes \Fil^0\Dcrisp(T)$$ in the following way. Let  $T^*$ be the $\Z_p$-linear dual of $T$
and consider the natural pairing 
$$[\sim, \sim]: \mathbb{D}_{\cris,p}(T)\times \mathbb{D}_{\cris,p}(T^*(1)) \rightarrow \Z_p$$ which can be extended linearly to $$\Q_p(\mu_{p^n}) \otimes_{\Z_p} \mathbb{D}_{\cris,p}(T) \times \Q_p(\mu_{p^n})  \otimes_{\Z_p} \mathbb{D}_{\cris,p}(T^*(1)) \rightarrow \Q_p(\mu_{p^n}).$$

Choose a $\Z_p$-basis $\{u_1,...,u_{2g}\}$ of $\mathbb{D}_{\cris,p}(T) $ and let $\{u_1^\prime,...,u_{2g}^\prime\}$ be the dual basis of $\mathbb{D}_{\cris,p}(T^*(1)) $. For $i \in \{1,...,2g\}$, we write $\mathcal{L}_{T,p}^i:\HIwcyc(\Q_p,T) \rightarrow \mathcal{H}$ for the map obtained by composing $\mathcal{L}_{T,p}$ and the projection of $\mathcal{H} \otimes \mathbb{D}_{\cris,p}(T)$ to the $u_i^{th}$-component. Now let $\theta$ be a Dirichlet character of conductor $p^n$ and $G_n=\Gal(\Q_p(\mu_{p^n})/\Q_p)$. Then 
\begin{equation}
 \theta(\mathcal{L}_{T,p}^i(z)) =
    \begin{cases}
      \left[\mathrm{exp}_{p,0}^*(z),(1-p^{-1}\varphi^{-1})(1-\varphi)^{-1}u_i^\prime\right] & \text{if $n=0$,}\\
      \tau(\theta^{-1})^{-1}\left[\sum_{\sigma \in G_n} \theta^{-1}(\sigma)\mathrm{exp}_{p,n}^*(z^\sigma),\varphi^{-n}(u_i^\prime)\right] & \text{otherwise}
    \end{cases}  
\end{equation}
where $\tau(\theta^{-1})$ denotes the Gauss sum of $\theta^{-1}$ (cf. \cite[p. 366]{BL17}).

In \cite[Theorem 1.1]{BL17}, Büyükboduk--Lei showed that the big logarithm map decomposes into Coleman maps in the following way.
$$
\mathcal{L}_{T,p}^{\infty} = (u_1,\ldots, u_{2g}) \cdot M_{T,p} \cdot \begin{bmatrix} \col_{T,p, 1} \\ \vdots \\ \col_{T,p,2g}\end{bmatrix}
,$$
where $M_{T,p}=\underset{n \rightarrow \infty}{\lim} M_{p,n}$ is a $2g \times 2g$  logarithmic matrix defined over $\calH$  and for $i \in \{1,...,2g\}$, the maps $\col_{T,p, i}, $ are $\Z_p[\Delta][[\Gamma]]$-homomorphisms from $\HIwcyc(\Q_p,T) \rightarrow \Z_p[\Delta][[\Gamma]].$
For any subset $I_p$ of $\{1,...,2g\}$, one can define
\begin{align*}
    \col_{T,I_p}: \HIwcyc(\Q_p,T) &\rightarrow \prod_{i=1}^{|I_p|}\Z_p[\Delta][[\Gamma]],\\
    {z} &\mapsto (\col_{T,p,i}({z}))_{i \in I_p}.
\end{align*}
Let $\calI$ be the set of all subsets of $\{1,...,2g\}$ of cardinality $g$.  So $\uI\in \calI$ will be a subset of $\{1,...,2g\}$ having cardinality $g$. Given any such $\uI \in \calI$  and $\mathbf{z}=z_1 \wedge \cdots \wedge z_{g} \in \bigwedge^g \HIwcyc(\Q_p,T),$ one can define 
$$\col_{T,\uI}(\mathbf{z})=\det(\col_{T,p,i}(z_j))_{i \in \uI,1 \leq j \leq g}.$$
For all $n \geq 1$, let $H_{p,n}=C_{p,n}\cdots C_{p,1}$.  For a pair $\uI$ and $\uJ$ in $\calI$, we define $H_{\uI,\uJ,n}$ as the $(\uI,\uJ)$-minor of $H_{p,n}$.

Since the prime $p$ is totally ramified in $\Q(\mu_{p^\infty})$, we will also use it to  denote the unique prime of $\Q(\mu_{p^\infty})$ above $p$. Let $I_p$ be a subset of of $\{1,...,2g\}$. One defines the submodule $$H^1_{I_p}(\Q(\mu_{p^\infty})_p, A^\vee[p^\infty]) \subset H^1(\Q(\mu_{p^\infty})_p, A^\vee[p^\infty])$$ as the orthogonal 
complement of $\ker(\col_{T,I_p})$ under the Tate's local pairing
$$H^1(\Q(\mu_{p^\infty})_p, A^\vee[p^\infty]) \times \HIwcyc(\Q_p,T) \rightarrow \Q_p/\Z_p.$$
By the inflation-restriction exact sequence it is possible to show that the restriction map 
$$H^1(\Q_{\infty,p}, A^\vee[p^\infty]) \rightarrow  H^1(\Q(\mu_{p^\infty})_p, A^\vee[p^\infty])^\Delta$$ is an isomorphism; because the order of $\Delta$ is $p-1$ and $A^\vee[p^\infty](\Q(\mu_{p^\infty})_p)$ is a finite $p$-group. Using this isomorphism one defines $$H^1_{I_p}(\Q_{\infty,p},A^\vee[p^\infty]) \subset H^1(\Q_{\infty,p},A^\vee[p^\infty]) $$
as $$H^1_{I_p}(\Q_{\infty,p},A^\vee[p^\infty])  = H^1_{I_p}(\Q(\mu_{p^\infty})_p,A^\vee[p^\infty])^\Delta. $$
\begin{definition}
    For $\uI =(I_p) \in \calI$, the $\uI$-Selmer group of $A^\vee$ over $\Q_\infty$ is defined by 
    $$\Sel_{\uI}(A^\vee/\Q_{\infty})=\ker\Big(H^1(\Q_{\infty}, A^\vee[p^\infty])  \rightarrow \prod_{v \nmid p}{H^1(\Q_{{\infty},v},A^\vee[p^\infty])} \times \frac{H^1(\Q_{{\infty},p},A^\vee[p^\infty])}{H^1_{I_p}(\Q_{{\infty},p},A^\vee[p^\infty])}\Big). $$
\end{definition}
\begin{remark}
    This definition of multi-signed Selmer groups can be generalized for any abelian variety over a general number field $F$, unramified at $p$, with good supersingular reduction at every prime $v$ of $F$ dividing $p$. In this case, 
    the multi-signed Selmer groups are defined for $\underline{I}=(I_v)_{v \mid p}$ where $I_v \subset \{1,...,2g[F_v:\Q_p]\}$. This is the notation used in page 5 of \cite{LP20}. In our case, since the base field is $\Q$, we have only one prime $p$ and to be consistent with the existing notation, we have written $\uI =(I_p)$. In our case $\uI$ and $I_p$ are the same.
\end{remark}

\section{Signed Mordell-Weil group and signed Tate-Shafarevich group}\label{sec:signed_MW}
In this section, we define the signed Mordell-Weil group and the signed Tate-Shafarevich group generalizing \cite{Lei2023plus_minus_MW}. 
\begin{definition}
The $\uI$-Mordell-Weil group of $A^\vee$ over the cyclotomic $\Z_p$-extension $\Q_\infty$ is defined by
$$\calM_{\uI}(A^\vee/\Q_\infty)=\ker\Big(A^\vee(\Q_\infty) \otimes \Q_p/\Z_p \rightarrow \frac{H^1(\Q_{\infty,p},A^\vee[p^\infty])}{H^1_{I_p}(\Q_{\infty,p},A^\vee[p^\infty])}\Big).$$
\end{definition}

Now, let us define the $\uI$-Mordell-Weil group at the level $\Q_{(n)}$, for which we need to define the local condition at $p$. 
Since $A^\vee(\Q_\infty)[p^\infty]=0$ (see \cite[Lemma 1.1]{LP20})
there is an identification 
$$H^1(\Q_{(n),p}, A^\vee[p^\infty]) \cong H^1(\Q_{\infty,p},A^\vee[p^\infty])^{\Gamma_n}.$$ For $n\geq 0$, we define $$H^1_{I_p}(\Q_{(n),p},A^\vee[p^\infty]):=H^1_{I_p}(\Q_{\infty,p},A^\vee[p^\infty])^{\Gamma_n}$$
and under the above identification this may be viewed as a subgroup of $H^1(\Q_{(n),p}, A^\vee[p^\infty])$.
As mentioned before, this definition is in coherence with \cite[Section 1.6]{ponsinet} and \cite[Definition 5.2]{Lei2023plus_minus_MW}. 
\begin{definition}
We define the $\uI$-Mordell-Weil group and the $\uI$-Selmer group over $\Q_{(n)}$ as 
\begin{align*}
    \calM_{\uI}(A^\vee/\Q_{(n)})&=\ker\Big(A^\vee(\Q_{(n)}) \otimes \Q_p/\Z_p \rightarrow \frac{H^1(\Q_{{(n)},p},A^\vee[p^\infty])}{H^1_{I_p}(\Q_{{(n)},p},A^\vee[p^\infty])}\Big),\\
    \Sel_{\uI}(A^\vee/\Q_{(n)})&=\ker\Big(H^1(\Q_{(n)}, A^\vee[p^\infty])  \rightarrow \prod_{v \nmid p}{H^1(\Q_{{(n)},v},A^\vee[p^\infty])} \times \frac{H^1(\Q_{{(n)},p},A^\vee[p^\infty])}{H^1_{I_p}(\Q_{{(n)},p},A^\vee[p^\infty])}\Big).    \end{align*}
\end{definition}
\begin{remark}\label{rem:imp}
When $A$ is an elliptic curves over $\Q$ with $a_p=0$, for $\uI=\{1\}$ and $\uI=\{2\}$, these definitions of $\Sel_{\uI}(A^\vee/\Q_{\infty})$ coincide with Kobayashi's plus and minus Selmer groups (see \cite[Appendix A]{BL17}). However, the definitions of $ \Sel_{\uI}(A^\vee/\Q_{(n)})$ does not coincide with Kobayashi's plus and minus Selmer groups \cite{kobayashi03} defined using $E^\pm(\Q_{(n),p}) \otimes \Q_p/\Z_p$ (see \cite[Remark 5.4]{Lei2023plus_minus_MW}). 
\end{remark}

\noindent \textbf{Open question.}  In \cite{kobayashi03}, Kobayashi could explicitly 
describe the local conditions $E^\pm(\Q_{(n),p})$ using trace maps. It is unknown whether one can write $H^1_{I_p}(\Q_{{(n)},p},A^\vee[p^\infty])=A^{\vee, \uI}(\Q_{{(n)},p}) \otimes \Q_p/\Z_p$ for some well-defined $A^{\vee, \uI}$.

\vspace{.2cm}

 \noindent Via the Kummer map, we  identify $A^\vee(\Q_{(n)})$ $\otimes$ $ \Q_p/\Z_p$ with a subgroup of  $H^1(\Q_{(n)}, A^\vee[p^\infty])$ and see that $\calM_{\uI}(A^\vee/\Q_{(n)})$ is a subgroup of $\Sel_{\uI}(A^\vee/\Q_{(n)})$. This allows us to define the $\uI$-Tate-Shafarevich group as follows.
 \begin{definition}
     Let $K=\Q_{(n)}$ or $\Q_\infty$. The ($p$-primary) $\uI$-Tate Shafarevich group ${\FSha}_{\uI}(A^\vee/K)$ is defined as the quotient $${\FSha}_{\uI}(A^\vee/K):=\Sel_{\uI}(A^\vee/K)/\calM_{\uI}(A^\vee/K).$$
 \end{definition}
 By definition, we have an exact sequence 
 $$0\rightarrow \calM_{\uI}(A^\vee/K) \rightarrow \Sel_{\uI}(A^\vee/K) \rightarrow {\FSha}_{\uI}(A^\vee/K) \rightarrow 0.$$

 \begin{theorem}\label{thm:I_MW}
 \begin{enumerate}[(i)]
     \item The map $$\Sel_{\uI}(A^\vee/\Q_{(n)}) \rightarrow \Sel_{\uI}(A^\vee/\Q_\infty)^{\Gamma_n}$$ is injective with bounded cokernel as $n$ varies.
     \item  The map $$m_{n,\uI}:\calM_{\uI}(A^\vee/\Q_{(n)}) \rightarrow \calM_{\uI}(A^\vee/\Q_\infty)^{\Gamma_n}$$
 is injective {and if ${\FSha}_{\uI}(A^\vee/\Q_{(n)})$ is finite}, then the map $m_{n,\uI}$ has finite cokernel. 
 \end{enumerate}

\end{theorem}
\begin{proof}
  The proof of part (i) essentially follows from the control theorem for fine Selmer group for places not above $p$ and the analysis of the local conditions at $p$. For completeness we reproduce the proof.  
    We consider the following commutative diagram.
{\small
\begin{equation}\label{diag: Sel_diagram1}
	\begin{tikzcd}
	0 \arrow[r] & \Sel_{\uI}(A^\vee/\Q_{(n)}) \arrow[d, "\alpha_{(n)}"] \arrow[r] &  H^1(\Q_{(n)}, A^\vee[p^\infty])  \arrow[d, "\beta_{(n)}"]  \arrow[r, "\theta_{(n)}"]  &\prod_{v \nmid p}{H^1(\Q_{{(n)},v},A^\vee[p^\infty])} \times \frac{H^1(\Q_{{(n)},p},A^\vee[p^\infty])}{H^1_{I_p}(\Q_{{(n)},p},A^\vee[p^\infty])} \arrow[d, "\delta_{(n)}=\prod_v\delta_{{(n)},v}"] \\
		0 \arrow[r] & \Sel_{\uI}(A^\vee/\Q_{\infty})^{\Gamma_n} \arrow[r] & H^1(\Q_{\infty}, A^\vee[p^\infty])^{\Gamma_n} \arrow[r, "\lambda_{\infty}^{\Gamma_n}"] &  \prod_{w \nmid p}{H^1(\Q_{{\infty},w},A^\vee[p^\infty])}^{\Gamma_n} \times \Big(\frac{H^1(\Q_{{\infty},p},A^\vee[p^\infty])}{H^1_{I_p}(\Q_{{\infty},p},A^\vee[p^\infty])}\Big)^{\Gamma_n}. 
	\end{tikzcd}
	\end{equation}    
 }
 It follows from \cite[Lemma 1.1]{LP20} that $A^\vee(\Q_\infty)[p^\infty]$ is trivial  and hence the inflation-restriction exact sequence gives that the central vertical map $\beta_{(n)}$ is an isomorphism.

For non-archimedean primes $v \nmid p$, the kernel of the map $$H^1(\Q_{{(n)},v},A^\vee[p^\infty]) \xrightarrow{\delta_{(n),v}}{H^1(\Q_{{\infty},w},A^\vee[p^\infty])}^{\Gamma_n}$$ is 
$H^1(\Gamma_n,A^\vee(\Q_{\infty,w})[p^\infty])\cong (A^\vee(\Q_{\infty,w})[p^\infty])_{\Gamma_n}.$ This is
finite and bounded by $$[A^\vee(\Q_{\infty,w})[p^\infty]: (A^\vee(\Q_{\infty,w})[p^\infty])_{\mathrm{div}}]$$ which is independent of $n$ (see \cite[Proof of Lemma 2.3]{LP20}). 
If  $A^\vee$ has good reduction at $v$, then $A^\vee(\Q_{\infty,w})[p^\infty]$ is divisible and hence $\ker(\delta_{(n),v})$ is trivial. Furthermore, there are finitely many non-archimedean primes where $A^\vee$ has bad reduction, and for each of them $\ker(\delta_{(n),v})$ is bounded and the bound is independent of $n$.

If $v$ is archimedean, $v$ splits completely in $\Q_\infty/\Q$ and hence $\ker(\delta_{(n),v})$ is trivial.
We are left to analyze the kernel of the map $\delta_{(n),p}$. Consider the diagram
{
\begin{equation*}
	\begin{tikzcd}
	0 \arrow[r] & H^1_{I_p}(\Q_{{(n)},p},A^\vee[p^\infty]) \arrow[d] \arrow[r] &  H^1(\Q_{{(n)},p},A^\vee[p^\infty])  \arrow[d]  \arrow[r]  & \frac{H^1(\Q_{{(n)},p},A^\vee[p^\infty])}{H^1_{I_p}(\Q_{{(n)},p},A^\vee[p^\infty])} \arrow[r] \arrow[d, "\delta_{(n),p}"] & 0 \\
		0 \arrow[r] & H^1_{I_p}(\Q_{{\infty},p},A^\vee[p^\infty])^{\Gamma_n} \arrow[r] & H^1(\Q_{{\infty},p},A^\vee[p^\infty])^{\Gamma_n} \arrow[r] &  \Big(\frac{H^1(\Q_{{\infty},p},A^\vee[p^\infty])}{H^1_{I_p}(\Q_{{\infty},p},A^\vee[p^\infty])}\Big)^{\Gamma_n}. 
	\end{tikzcd}
	\end{equation*}    
 }

By the definition of the local condition $H^1_{I_p}(\Q_{{(n)},p},A^\vee[p^\infty])$, the left vertical map is an isomorphism. The central vertical map has kernel $$H^1(\Gamma_n, A^\vee(\Q_{\infty,p})[p^\infty])$$
which is trivial since the group $A^\vee(\Q_{\infty,p})$ has no $p$-torsion (see  \cite[Lemma 1.1]{LP20}). Therefore, by the snake lemma the kernel of $\delta_{(n),p}$ is trivial.

Part (ii) follows by applying snake lemma to the following diagram and using part (i).

\begin{equation*}
	\begin{tikzcd}
	0 \arrow[r] & \calM_{\uI}(A^\vee/\Q_{(n)})\arrow[d] \arrow[r] &  \Sel_{\uI}(A^\vee/\Q_{(n)}) \arrow[d]  \arrow[r]  & \FSha_{\uI}(A^\vee/\Q_{(n)})\arrow[r] \arrow[d] & 0 \\
		0 \arrow[r] & \calM_{\uI}(A^\vee/\Q_{\infty})^{\Gamma_n} \arrow[r] & \Sel_{\uI}(A^\vee/\Q_{\infty})^{\Gamma_n} \arrow[r] &  \FSha_{\uI}(A^\vee/\Q_{\infty})^{\Gamma_n}. 
	\end{tikzcd}
	\end{equation*}   

\end{proof}

Let us identify $\Z_p[[\Gamma]]$ with $\Z_p[[X]]$ via $X=\gamma -1$ where $\gamma$ is a topological generator of the Galois group $\Gamma$. Define $\Phi_0=X$ and for $n \geq 1$, we write $\Phi_n=\frac{(X+1)^{p^n}-1}{(X+1)^{p^{n-1}}-1}$ 
for the $p^n-$th cyclotomic polynomial in $X+1$. Let $\omega_n=(X+1)^{p^n}-1$.
It follows from \cite[Theorem 2.1.2]{Lee20} that there is a pseudo-isomorphism 
\begin{equation}\label{eq:Lee}
(A^\vee(\Q_\infty) \otimes_{\Z_p} \Q_p/\Z_p)^{\land} \sim \Lambda^r \oplus \Big(\bigoplus_{i=1}^t\frac{\Lambda}{\Phi_{b_i}}\Big),
\end{equation}
for certain non-negative integers $r,t$ and $b_i$. 
\begin{proposition}\label{thm:calM_I}
  If $r=0$, then there is a pseudo-isomorphism of $\Lambda$-modules
  $$\calM_{\uI}(A^\vee/\Q_\infty)^{\land} \sim  \bigoplus_{i=1}^{u_{\uI}}\frac{\Lambda}{\Phi_{c_{i,\uI}}} $$
  for some non-negative integers $u_{\uI}$ and $c_{i,\uI}$. 
\end{proposition}
\begin{proof}
    By definition, the Pontryagin dual of $\calM_{\uI}(A^\vee/\Q_\infty)$ is a quotient of the dual of $A^\vee(\Q_\infty) \otimes_{\Z_p} \Q_p/\Z_p$. If $r=0$, the right-hand side of \eqref{eq:Lee} is a direct sum of all simple modules. Hence it follows that $\calM_{\uI}(A^\vee/\Q_\infty)^{\land}$ is pseudo-isomorphic to a partial direct sum of the right-hand side of \eqref{eq:Lee}.
\end{proof}
Let $\Sigma$ be a finite set of primes of $\Q$ containing the prime $p$, the archimedean prime and the primes of bad reduction of $A^\vee$. Let $\Q_{\Sigma}$ be the maximal extension of $\Q$ unramified outside $\Sigma$. For $i \geq 0$, let $H^i_{\mathrm{Iw},\Sigma}(\Q,T)$ be the projective limit of the groups $H^i(\Q_{\Sigma}/\Q_{(n)},T)$ relative to the corestriction maps. We assume the following hypothesis.

\vspace{.2cm}

\noindent \textbf{(Torsion)} For all $\uI \in \calI$, the $\Lambda$-module $\Sel_{\uI}(A^\vee/\Q_\infty)$ is cotorsion.
\vspace{.2cm}

This is known to be true for elliptic curves over $\Q$ with non-ordinary reduction at $p$ with $a_p=0$ (see \cite{kobayashi03}).

Under the assumption \textbf{(Torsion)}, $H^1_{\mathrm{Iw},\Sigma}(\Q,T)$ is a $\Z_p[[\Gamma]]$-module of rank $g$ (see \cite[Lemma 2.4]{LP20}). Hence we can fix a family of classes $c_1,...,c_g \in H^1_{\mathrm{Iw},\Sigma}(\Q,T)$ such that $$\frac{H^1_{\mathrm{Iw},\Sigma}(\Q,T)}{\langle c_1,...,c_g \rangle} \text{ is } \Z_p[[\Gamma]]\text{-torsion}.$$
Let $J_p$ be a subset of $\{1,...,2g\}$ of cardinality $g$. The composition map
$$H^1_{\mathrm{Iw},\Sigma}(\Q,T) \xrightarrow{\mathrm{loc}_p} \HIwcyc(\Q_p,T)^\Delta \xrightarrow{\col_{T,J_p}} \prod_{k=1}^g \Z_p[[\Gamma]]$$
is a $\Z_p[[\Gamma]]$-homomorphism between two $\Z_p[[\Gamma]]$-modules of rank $g$. For $\uJ=(J_p)$, we write
$$\col_{T,\uJ}(\mathbf{c})=\det(\col_{T,J_p} \circ  \text{ } \mathrm{loc}_p(c_i))_{1 \leq i \leq g}.$$

Under the assumption that the Selmer group $\Sel_{\uJ}(A^\vee/\Q_\infty)$ is $\Z_p[[\Gamma]]$-cotorsion, Lei--Ponsinet showed that $\col_{T,\uJ}(\mathbf{c})\neq 0$ (see  \cite[Lemma 3.2]{LP20}).

\subsection{Change of basis}

	The definitions of $\calM_{\uI}(A^\vee/\Q_\infty)$ and $\Sel_{\uI}(A^\vee/\Q_\infty)$ (and hence that of $\FSha_{\uI}(A^\vee/\Q_\infty)$) depend on the choice of a Hodge-compatible basis and hence are not canonically defined. Under certain additional assumptions, one can define them canonically for certain choices of the indexing set $\uI$.
Let  $\uI_1=\{g+1,...,2g\}$ and recall that $\uI_0=\{1,...,g\}$. 
\begin{proposition}\label{prop:canonicaldefn}
	Suppose that $C_p$ is a block anti-diagonal matrix with respect to the basis $\mathcal{B}_p$. For $\uI=\uI_0$ or $ \uI_1$, the groups $\calM_{\uI}(A^\vee/\Q_\infty)$, $\Sel_{\uI}(A^\vee/\Q_\infty)$ and $\FSha_{\uI}(A^\vee/\Q_\infty)$ are canonically defined.
\end{proposition}
\begin{proof}

	Let $\mathcal{B}_p=\{u_1,\ldots,u_{2g}\}$ and  $\mathcal{B}_p^\prime=\{w_1,\ldots,w_{2g}\}$ be a pair of Hodge-compatible bases of $\mathbb{D}_{\cris,p}(T)$. Let $B_p$ be the change of basis matrix from $\mathcal{B}_p^\prime$ to $\mathcal{B}_p$. This implies that  $B_p$ is  block diagonal (see \cite[p. 371]{BL17}). Write
	$$
	B_p = \left[ \begin{array}{c|c}
	B_{1,1} & 0 \\
	\hline
	0 & B_{2,2}
	\end{array} \right]
	$$
	where $B_{1,1},B_{2,2} \in \GL_g(\Z_p)$. Let $\col^{\mathcal{B}_p}_{T,p,i}$ be the $i$-th Coleman map with respect to the basis $\mathcal{B}_p$. Let $\col^{\mathcal{B}_p}_{T,p}$ denote the column vector of Coleman maps $( \col^{\mathcal{B}_p}_{T,p,i} )_{i=1}^{2g}$. Similarly define $\col^{\mathcal{B}_p^\prime}_{T,p}$ as the column vector of Coleman maps defined with respect to the basis $\mathcal{B}_p^\prime$. An argument analogous to \cite[Proposition 4.18]{DionRay}  gives
	$
	\col^{\mathcal{B}_p^\prime}_{T,p}= B_p\cdot  \col^{\mathcal{B}_p}_{T,p}.
	$
	Thus,
	$$
	\begin{bmatrix} \col_{T,p,1}^{\mathcal{B}_p^\prime} \\ \vdots \\ \col_{T,p,g}^{\mathcal{B}_p^\prime} \end{bmatrix}=B_{1,1} \begin{bmatrix} \col_{T,p,1}^{\mathcal{B}_p} \\ \vdots \\ \col_{T,p,g}^{\mathcal{B}_p} \end{bmatrix}, \quad \begin{bmatrix} \col_{T,p,{g+1}}^{\mathcal{B}_p^\prime} \\ \vdots \\ \col_{T,p,{2g}}^{\mathcal{B}_p^\prime} \end{bmatrix}=B_{2,2} \begin{bmatrix} \col_{T,p,{g+1}}^{\mathcal{B}_p} \\ \vdots \\ \col_{T,p,{2g}}^{\mathcal{B}_p} \end{bmatrix}.
	$$
	 One obtains that $(\col_{T,p,i}^{\mathcal{B}_p^\prime}(z))_{i=1}^g=0$ if and only if $(\col_{T,p,i}^{\mathcal{B}_p}(z))_{i=1}^g=0$ since the matrix $B_{1,1}$ is invertible. The same argument also works  for $(\col_{T,p,i}^{\mathcal{B}_p^\prime}(z))_{i=g+1}^{2g}$ and $(\col_{T,p,i}^{\mathcal{B}_p}(z))_{i=g+1}^{2g}$. Therefore $\ker \col_{T,I_p}$ is independent of the choice of basis if $I_p=\{1,...,g\}$ or $\{g+1,...,2g\}$. 
\end{proof} 
\medskip

	\subsection{Remarks and examples}\label{sec:examples}
		~
		\begin{itemize}
			\item Note that $C_p$ will remain block anti-diagonal upon a change of basis. This is because $B_p C_{p}B_{p}^{-1}$ is block anti-diagonal if $B_p$ is block diagonal and $C_p$ is block anti-diagonal. \smallbreak
   \item In the case of elliptic curves with $a_p=0$, note that Kobayashi's signed Selmer groups $\Sel_p^\pm(E/\Q_\infty)$ are canonically defined. These groups correspond to $\Sel_{\uI_0}(A^\vee/\Q_\infty)$ and $\Sel_{\uI_1}(A^\vee/\Q_\infty)$ for $g=1$ (see \cite[Appendix 4]{BL17}). Hence Proposition \ref{prop:canonicaldefn} should be seen as a generalization of this phenomenon. Kobayashi showed that the Selmer groups $\Sel_p^\pm(E/\Q_\infty)$ are $\Z_p[[\Gamma]]$-cotorsion and hence the assumption \textbf{(Torsion)} is  satisfied. Here $r=0$ (cf. \cite[Remark 3.5]{Lei2023plus_minus_MW}) and we know that  rank of $E(\Q_{(n)})$ is bounded as $n$ varies (cf. \cite[discussion after eq. (1.1)]{Lei2023plus_minus_MW}).\smallbreak
			\item Note that in the context of abelian varieties of  $GL(2)$-type, under some additional hypotheses, one can ensure that   the matrix of Frobenius $C_p$ is block anti-diagonal  (see \cite[Section 3.3]{LP20}). 
		\end{itemize}
        \section{Questions on the growth of signed structures}
        \subsection{ Some preliminaries}
      \begin{definition}\label{defn:kobayashi_rank}
	Let $(N_n)_n$ be an inverse system of finitely generated $\Z_p$-modules with transition maps $\pi_n:N_n \rightarrow N_{n-1}$. Suppose that the map $\pi_n$ has finite kernel and cokernel. Then the Kobayashi rank $\nabla N_n$ is defined as 
	$$\nabla N_n:=\mathrm{length}_{\Z_p}(\ker \pi_n) - \mathrm{length}_{\Z_p}(\coker \pi_n) + \rank_{\Z_p}N_{n-1}.$$
\end{definition}
Here are a few basic properties that we will need. For clarifications see \cite[Lemma 10.5 (ii)]{kobayashi03} and \cite[Lemma 4.3]{LLZ2017}.

\begin{lemma}\label{lem:short_exact}
Let $0 \rightarrow (N_n^\prime) \rightarrow (N_n) \rightarrow (N_n^{\prime\prime}) \rightarrow 0$ be a short exact sequence. If any two of $\nabla N_n, \nabla N_n^\prime, \nabla N_n^{\prime \prime}$ are defined, then the other is also defined and $$\nabla N_n=\nabla N_n^\prime + \nabla N_n^{\prime \prime}.$$ 
\end{lemma}
We write $\omega_n$ for $\omega_n(X):=(1+X)^{p^n}-1$.
\begin{lemma}\label{lem:2.4}
	Let $N$ be a finitely generated torsion $\Z_p[[X]]$-module with the characteristic polynomial $f$, and let $N_n=N/\omega_nN$. Consider the projective system $(N_n)_{n}$. Then for $n \gg 0$, $\nabla N_n$ is defined and 
	$$\nabla N_n = \lambda(N) + (p^n-p^{n-1})\mu(N),$$
 where $\lambda$ and $\mu$ are the Iwasawa $\lambda$- and $\mu$-invariants of $N$.
	\begin{lemma}
		Suppose that $(N_n)_n$ is an inverse system of finitely generated $\Z_p$-modules such that for all $n \geq 1$, $|N_n|=p^{s_n}$ for some integer $s_n$. Then $$\nabla N_n=s_n-s_{n-1}.$$
	\end{lemma}
\end{lemma}

\textbf{Question 1.} This question is to study the relation between $\FSha_{\uI}(A^\vee/\Q_{(n)})$ and the classical Tate-Shafarevich group $\Sha(A^\vee/\Q_{(n)}).$ If one removes all the local conditions from the definitions of $\calM_{\uI}(A^\vee/\Q_{(n)}) $ and $\Sel_{\uI}(A^\vee/\Q_{(n)}) $, then one obtains the classical fine Mordell-Weil group $\calM(A^\vee/\Q_{(n)}) $ and the fine Selmer group $\Sel_0(A^\vee/\Q_{(n)}) $ defined by Wuthrich \cite{Wuthrich2007fine}. Their quotient $\FSha(A^\vee/\Q_{(n)}):=\frac{\Sel_0(A^\vee/\Q_{(n)})}{\calM(A^\vee/\Q_{(n)})}$ is then the classical fine Tate-Shafarevich group and Wuthrich in \textit{(loc.cit)} has shown that $\FSha(A^\vee/\Q_{(n)})$ is a subgroup of  $\Sha(A^\vee/\Q_{(n)})[p^\infty]$. 

But the relation between the corresponding signed version $\FSha_{\uI}(A^\vee/\Q_{(n)})$ and $\Sha(A^\vee/\Q_{(n)})[p^\infty]$ is unclear. 

Another related question is whether $\FSha_{\uI}(A^\vee/\Q_{(n)})$ is (at least) conjecturally finite for all $n$.

 \vspace{.2cm}
Having answered question (1), here is another question.
        
      \vspace{.2cm}

\textbf{Question 2.} Suppose $E$ is an elliptic curve over $\Q$ with supersingular reduction at $p$ and $a_p=0$. Then following the second bullet point of remark  \ref{sec:examples} we know that rank $E(\Q_{(n)})$ is bounded as $n$ varies. So by definition $\nabla \big(E(\Q_{(n)}) \otimes \Q_p/\Z_p\big)^{\wedge}$ is ${rank} \text{ }E(\Q_\infty)$ where $(-)^\wedge$ denotes Pontryagin dual. 

\vspace{.2cm}
  Let $\uI_0=\{1,...,g\} \in \calI$. Let $\theta$ be a character on $\Gamma$  of conductor $p^{n+1}$ which is trivial on $\Gal(\Q(\mu_{p^\infty})/\Q_\infty)$. Suppose that 
\begin{equation}\label{cond:character}
\sum_{\uJ}(H_{{\uI}_0,\uJ,n}\col_{T,\uJ}(\textbf{c}))(\theta) \neq 0
    \end{equation}
    for $n \gg 0$ where the matrix $H_{{\uI}_0,\uJ,n}$ is as in \cite[Prop. 3.3]{LP20}. Then Lei--Ponsinet showed that the $\Z_p$-rank of $A^\vee(\Q_{(n)})$ is bounded as $n$ varies (see \cite[Theorem 3.4]{LP20}).
    When the matrix of Frobenius $C_p$ is block anti-diagonal  (see \cite[Section 3.3]{LP20}) and under the assumptions that $\Sel_{\uI_0}(A^\vee/\Q_\infty)$ and $\Sel_{\uI_1}(A^\vee/\Q_\infty)$ are $\Z_p[[\Gamma]]$-cotorsion, one obtains  examples when condition \eqref{cond:character} is true (see \cite[Lemma 3.6]{LP20}).

 Note that $A^\vee(\Q_{(n)}) \otimes \Q_p/\Z_p$ is a divisible group. There is an injection $$\calM_{\uI}(A^\vee/\Q_{(n)}) \rightarrow A^\vee(\Q_{(n)} \otimes \Q_p/\Z_p$$ but $\calM_{\uI}(A^\vee/\Q_{(n)}) $ is not necessarily divisible \cite[Remark 3.4]{Lei2023plus_minus_MW}.
 
 \vspace{.2cm}
 
 \textit{An example:} Wuthrich has shown that for the elliptic curve defined by $E: y^2+y=x^3-x$ and prime $p=179$ (at which $E$ has good, ordinary reduction) the classical fine Mordell-Weil group $\mathcal{M}(E/\Q)$ is a finite non-trivial group and more specifically it is isomorphic to $ \Z/p\Z$ (cf. \cite[Section 7]{Wuthrich2007fine}).

\vspace{.2cm}

 It will be interesting to have several numerical examples of abelian varieties with good supersingular reduction at $p$, for which one can compute
$\calM_{\uI}(A^\vee/\Q_{(n)})$ for all $n$ and explicitly see how it grows as $n \rightarrow \infty$.
Another question is to try to compute the growth of the $\uI$-Tate-Shafarevich groups $\FSha_{\uI}(A^\vee/\Q_{(n)})$ as $n$-varies.

\vspace{.2cm}

\begin{remark}
Assuming $\FSha_{\uI}(A^\vee/\Q_{(n)})$ is finite for all $n$, the answer to compute the growth of the $\uI$-Tate-Shafarevich group $\FSha_{\uI}(A^\vee/\Q_{(n)})$ in question (2) above (in a general framework) has been proposed by Meng Fai Lim recently in the following way. Taking the control theorem (cf. Theorem \ref{thm:I_MW}) into consideration, one may apply the argument similar to that in \cite[Proposition 4.1]{Lim2020} to show that the sequence 
$$0 \rightarrow \FSha_{\uI}(A^\vee/\Q_{\infty})^\wedge \rightarrow \mathfrak{G}\left(\Sel_{\uI}(A^\vee/\Q_{\infty})^\wedge\right) \rightarrow \mathfrak{G}\left(\mathcal{M}_{\uI}(A^\vee/\Q_{\infty})^\wedge\right)\rightarrow 0$$
is exact.  Here $\mathfrak{G}$ is defined as in the sense of Lee \cite{Lee20} (also see \cite[Section 2]{Lim2020}). In particular, by \cite[Prop. 2.4 (1)]{Lee20}, $\FSha_{\uI}(A^\vee/\Q_{\infty})^\wedge $ is $\Z_p[[\Gamma]]$-torsion.

Assume that $A^\vee(\Q_{(n)})$ is bounded as $n$ varies. Taking this boundedness of Mordell-Weil rank into account, the argument in \cite[Theorem 4.6]{Lim2020} will give us an asymptotic formula for the growth of the $p$-exponent of  $\FSha_{\uI}(A^\vee/\Q_{(n)})$ as $n$-varies. More precisely, one obtains,
$$e\left(\FSha_{\uI}(A^\vee/\Q_{(n)})\right)=\mu\left(\FSha_{\uI}(A^\vee/\Q_{\infty})^\wedge\right)p^n+\lambda\left(\FSha_{\uI}(A^\vee/\Q_{\infty})^\wedge\right)n+\nu$$
for $n \gg 0$ where $e\left(\FSha_{\uI}(A^\vee/\Q_{(n)})\right)=\ord_p(|\FSha_{\uI}(A^\vee/\Q_{(n)})|)$ and $\mu$- and $\lambda$- are classicial Iwasawa invariants of the torsion module $\FSha_{\uI}(A^\vee/\Q_{\infty})^\wedge $.

\end{remark}
\section*{Declarations}
\subsection*{Ethics declarations}

\subsubsection*{Competing interest.} The author has no competing interests to declare that are relevant to the content of this article.

\bibliographystyle{alpha}
\bibliography{main}
\end{document}